\documentclass[11pt]{amsart}
\usepackage{categorytheory}
\usepackage{diags}

\newcommand{\WaldCat}{\mathbf{WaldCat}}
\newcommand{\Cat}{\mathbf{Cat}}
\newcommand{\timeses}{\times\cdots\times}
\renewcommand{\smash}{\wedge}
\newcommand{\smashes}{\wedge\cdots\wedge}

\newcommand{\defeq}{\stackrel{\mathrm{def}}{=}}
\newcommand{\A}{\mathcal{A}}
\newcommand{\B}{\mathcal{B}}
\newcommand{\I}{\mathcal{I}}

\renewcommand{\dot}{{{\raisebox{.2ex}{\scalebox{.3}{$\bullet$}}}}}
\newcommand{\Ar}{\mathrm{Ar}}
\newcommand{\ev}{\mathrm{ev}}
\newcommand{\N}{\mathbf{N}}
\newcommand{\V}{\mathcal{V}}
\renewcommand{\M}{\mathcal{M}}
\newcommand{\uM}{\underline{\mathcal{M}}}
\newcommand{\uWaldCat}{\underline{\mathbf{WaldCat}}}
\newcommand{\<}{\langle}
\renewcommand{\>}{\rangle}

\DeclareMathOperator{\coeq}{coeq}
\DeclareMathOperator{\dom}{dom}

\newtheorem*{lemmapushout}{Lemma \ref{lem:pushout}}

\begin{document}
\title[WaldCat is closed multicategory]{The category of Waldhausen categories is a closed multicategory}
\author{Inna Zakharevich}
\maketitle

\tableofcontents

\section*{Introduction}

The goal of this paper is to develop in detail an example of a closed
multicategory.  The literature on closed multicategories has very few examples;
in this paper we aim to explain a potentially-useful example in enough detail
that both the example and the general theory are easier to understand.  This
paper was written because the details were necessary to the future work of
Angelica Osorno and Anna Marie Bohmann, but we hope that it will be useful for
others as well.  Very little in this paper is new, and it is particularly
indebted to \cite{blumbergmandell11} for many of the ideas.



This paper proves the following theorem:
\begin{theorem}
  The category $\WaldCat$ of Waldhausen categories is a closed symmetric
  multicategory, in the sense that the hom-sets
  \[\WaldCat_k(A_1,\ldots,A_k;B)\]
  all have the structure of Waldhausen categories and composition of morphisms
  is multiexact.  In addition, there is a multiexact
  \[\ev:\C_1\timeses \C_k\times \WaldCat_k(\C_1,\ldots,\C_k;\D) \rto \D\]
  defined by 
  \[(A_1,\ldots,A_k,F) \rgoesto F(A_1,\ldots,A_k).\]
\end{theorem}

The organization of this paper is as follows.  In Section~1 we introduce
multicategories and closed multicategories and in Section~2 we introduce
Waldhausen categories.  Sections~3 and 4 discuss $k$-exactness of functors.  The
definition of the hom-Waldhausen categories is given in Section~5, and the
analysis of the $K$-theory functor is in Section~6.

\section{A quick introduction to symmetric multicategories}

A multicategory is a generalization of a symmetric monoidal category where one
does not necessarily have a product.  The motivation for the definition comes
from the notion of tensor product: the tensor product of modules classifies
bilinear maps out of the ordinary product of the modules.  Thus if one is in a
context where the tensor product is difficult to work with directly, one can
work with bilinear maps instead.  The idea of a multicategory is that we have a
notion of ``$k$-linear'' map, but we do not necessarily have a representing
object, so we must always work with the ``$k$-linear'' maps directly.

More formally, we have the following definition \cite{elmendorfmandell}:
\begin{definition}
  A \textsl{symmetric multicategory} $\M$ is given by the following data:
  \begin{itemize}
  \item[-] A collection of objects $\ob\M$.
  \item[-] For each $k\geq 0$ and $k+1$-tuple of objects $A_1,\ldots,A_k,B\in
    \ob\M$, a set $\M_k(A_1,\ldots,A_k;B)$ of \textsl{$k$-morphisms}.
  \item[-] A right action of $\Sigma_k$ on the collection of all $k$-morphisms such
    that for $\sigma\in \Sigma_k$,
    \[\sigma^*:\M_k(A_1,\ldots,A_k;B) \rto
    \M_k(A_{\sigma(1)},\ldots,A_{\sigma(k)};B).\] 
  \item[-] A distinguished \textsl{unit} $1_A\in \M_1(A;A)$ for every $A\in \ob\M$,
    and 
  \item[-] A \textsl{composition law} 
    \[\begin{array}{l}
      \circ: \M_\ell(B_1,\ldots,B_\ell;C) \times \prod_{i=1}^\ell
      M_{k_i}(A_{i1},\ldots,A_{ik_i}; B_i) \\ \qquad\rto \M_{\sum
        k_i}(A_{11},\ldots,A_{\ell k_\ell};C).
    \end{array}\]
  \end{itemize}
  subject to compatibility axioms listed in \cite[p5-6]{elmendorfmandell}.  We
  do not state them here as we will need to restate them in the enriched setting
  momentarily. 
\end{definition}

Any symmetric monoidal category is a symmetric multicategory, by setting 
\[\M_k(A_1,\ldots,A_k;B) = \M(A_1\otimes\cdots\otimes A_k,B).\]
Thus a symmetric multicategory is the ``next best thing'' to a symmetric
monoidal category in many cases.  Many proofs that work in a symmetric monoidal
category where the product is defined via an appropriate universal property will
also work in a symmetric multicategory.

Quite often one wants more than just a symmetric monoidal category structure:
one also wants the symmetric monoidal category to be closed, so that there are
hom-objects defined in the category itself.  In other words, we want $\M$ to be
enriched over itself \cite[Section~1.6]{kelly82} (in a way compatible with the
symmetric monoidal structure).  We thus have the following definition, where we
will write $\bar A = (A_1,\ldots,A_k)$ in the interest of compactness.

\begin{definition} \lbl{def:enrichedmulti} A symmetric multicategory $\M$ is
  called \textsl{closed} if for all $k+1$-tuples $A_1,\ldots,A_k,B$ of objects
  in $\M$ there exists an object $\uM(\bar A;B)\in \M$ with a right
  $\Sigma_k$-action called the \textsl{internal hom-object} and an
  \textsl{evaluation morphism}
  \[\ev_{\bar A;B} \in
  \M_{k+1}(\bar A,\uM(\bar A;B);B).\] These need to satisfy the following
  axioms:

  \begin{itemize}
  \item[(CM1)] for all $\ell$-tuples $C_1,\ldots,C_\ell\in \ob\M$ there is a
    bijection
    \[
      \varphi_{\bar A;\bar C;B}:\M_{\ell}(\bar C; \uM(\bar A;B))   \rto
      \M_{k+\ell}(\bar A,\bar C; B)\]
    defined by sending  $f\in
    \M_\ell(\bar C;\uM(\bar A;B))$ to the composite 
    \[\ev_{\bar A;B} \circ (1_{A_1},\ldots,1_{A_k},f).\]
  \item[(CM2)]  This bijection is $\Sigma_k\times \Sigma_\ell$-equivariant, in
    the sense that the following diagram commutes for all $(\sigma,\tau)\in
    \Sigma_k\times \Sigma_\ell$:
    \begin{diagram}[4em]
      {\M_\ell(\bar C;\uM(\bar A;B)) & \M_{k+\ell}(\bar A,\bar C; B) \\
        \M_\ell(\bar C_\tau; \uM(\bar A_\sigma;B)) & \M_{k+\ell}(\bar
        A_\sigma,\bar C_\tau;B)\\}; \arrowsquare{\varphi_{\bar A;\bar
          C;B}}{f\mapsto \sigma \circ (f\cdot \tau)}{g\mapsto g\cdot (\sigma\times \tau)
        }{\varphi_{\bar A_\sigma;\bar C_\tau;B}}
    \end{diagram}
    Here $\bar A_\sigma = (A_{\sigma(1)},\ldots,A_{\sigma(k)})$ and $\bar C_\tau =
    (C_{\tau(1)},\ldots,C_{\tau(\ell)})$.
  \end{itemize}
  
\end{definition}
For more details, see \cite[Section 3]{manzyuk12}; for a more detailed theory of
enriched categories, see \cite{leinster04}.

For example, if $\M$ is a closed symmetric monoidal category then we can give it
the structure of a closed symmetric multicategory by setting
$\uM(A_1,\ldots,A_k;B) = B^{A_1\otimes\cdots\otimes A_k}$.  Note that if $\M$
is a closed symmetric multicategory then we can think of it as a category
enriched over itself.

\section{A bit about Waldhausen categories}

We begin by recalling the definition of a Waldhausen category.  These were first
introduced by Waldhausen in \cite{waldhausen83}, where they are called
``categories with cofibrations and weak equivalences.''

\begin{definition}
  A \textsl{Waldhausen category} is a category $\C$ together with two
  subcategories, $c\C$ and $w\C$, of cofibrations and weak equivalences,
  satisfying the following extra axioms:
  \begin{itemize}
  \item[(W1)] All isomorphisms are both weak equivalences and cofibrations.
  \item[(W2)] If two of $f,g,gf$ are weak equivalences, so is the third.
  \item[(W3)] $\C$ has a zero object, and the morphism $0\rto A$ is a cofibration
    for all $A\in \C$.
  \item[(W4)] Every diagram $C \lto A \rcofib B\in \C$ has a pushout, and the
    morphism $C \rcofib B\cup_A C$ is a cofibration.
  \item[(W5)] Given any diagram
    \begin{diagram}
      {B & A & C \\ B' & A' & C'\\};
      \cofib{1-2}{1-3} \cofib{2-2}{2-3}
      \to{1-2}{1-1} \to{2-2}{2-1}
      \we{1-1}{2-1} \we{1-2}{2-2} \we{1-3}{2-3}
    \end{diagram}
    the induced morphism $B\cup_A C \rwe B'\cup_{A'} C'$ is a weak equivalence.
  \end{itemize}
  A functor $F:\C\rto \D$ between Waldhausen categories $\C$ and $\D$ is
  \textsl{$1$-exact} if it preserves weak equivalences, cofibrations, and pushouts
  of the form described in (W4).
\end{definition}

Before we move on to a very quick overview of the $S_\dot$ construction for a
Waldhausen category, we introduce a couple of technical definitions which will
be of great use to us in the upcoming discussion.

\begin{definition}
  Let $\I$ be the category with two objects $0$ and $1$ and one non-invertible
  morphism $0\rto 1$.  Suppose that we are given two functors $F,G:\C\rto \D$
  and a natural transformation $\alpha:F\rto G$.  We write $(F\Rto^\alpha G)$ for
  the functor $\C\times\I\rto \D$ given by
  \[\makeshort{(F\Rto^\alpha G)}(A) = \begin{cases} F(A) \caseif A\in \C\times \{0\} \\ G(A)
    \caseif A\in \C\times \{1\}\end{cases}\]
  and
  \[\makeshort{(F\Rto^\alpha G)(f)}
  = \begin{cases}
    F(f) \caseif f\in \C\times\{0\} \\ G(f) \caseif f\in \C\times\{1\} \\
    \alpha_A \caseif f:\makeshort{(A,0) \rto (A,1)} \end{cases}\]
  Note that the existence of such a functor is equivalent to the existence of
  the natural transformation $\alpha$.
\end{definition}

\begin{definition}
  An \textsl{$n$-cube} in $\C$ is a functor $I:\I^n \rto \C$; a \textsl{face} of
  a cube is a restriction $I\big|_{\I^k\times\{\epsilon\}\times\I^{n-k-1}}$ for
  $\epsilon=0,1$.  We will write the objects of $\I^n$ as vectors $\bar \epsilon
  = (\epsilon_1,\ldots,\epsilon_n)$.  As shorthand, we will write
  $I_{k\epsilon}$ for the restriction $I\big|_{\I^{k-1}\times\{\epsilon\}\times
    \I^{n-k}}$, for $\epsilon=0,1$, and $I(1)$ for $I(1,\ldots,1)$.  For any
  cube $I$ we write $I'$ for the diagram $I\big|_{\bar\epsilon \neq
    (1,\ldots,1)}$, and define the \textsl{southern arrow} of an $n$-cube $I$
  to be the morphism
  \[\colim I' \rto I(1).\] (The southern arrow may not
  exist if the colimit does not.)  An $n$-cube $I$ is \textsl{good} if its
  southern arrow is a cofibration and all of its faces are good.  
\end{definition}

In particular, the $0$-cubes are the objects of $\C$, and the southern arrow of
a $0$-cube $A$ is just the morphism $\initial \rto A$, so all $0$-cubes are
good.  The $1$-cubes are the morphisms of $\C$, and the southern arrow of a
$1$-cube is itself.  Thus the good $1$-cubes are the cofibrations.

The notion of a good cube appears in \cite[Definition~2.1]{blumbergmandell11}
as a ``cubically cofibrant'' diagram.  

Given a natural transformation $\alpha:I \rto J$ between $n$-cubes,
we will write $[\alpha]$ for the $n+1$-cube $(I\Rto^\alpha J)$.

Let $\< n\>$ be the ordered set $0<1<\cdots<n$ and let $\Ar\< n\>$ be the arrow
category of $\< n\>$; we will denote an object in $\Ar\< n\>$ as $j<i$.  For a
vector $\vec n = (n_1,\ldots,n_m)$ we will write $\<\vec n\> = \< n_1\>\timeses
\< n_m\>$.

\begin{definition}
  The category $S_n\C$ is defined to have as objects functors $X:\Ar\<
  n\> \rto \C$ satisfying the extra conditions
  \begin{enumerate}
  \item $X(i=i) = *$ for all $i\in \< n \>$, and
  \item $X(i<j) \rto X(i<k)$ is a cofibration, and 
  \item for all $i<j<k$ the square
    \begin{diagram} {X(i<j) & X(i<k) \\ X(j=j) & X(j<k)\\}; 
      \cofib{1-1}{1-2} \cofib{2-1}{2-2}
      \to{1-1}{2-1} \to{1-2}{2-2}
    \end{diagram}
    is a pushout square.
  \end{enumerate}
  The categories $S_n\C$ form a simplicial category by inheriting the simplicial
  structure from the simplicial category $[\Ar \<\dot\>, \C]$.

  $S_n\C$ is defined to be a Waldhausen category by setting the weak
  equivalences to be levelwise, and the cofibrations to be the natural
  transformations $\alpha:X\rto Y$ such that for all $i<j$ the square
  \begin{diagram}
    {X(i) & X(j) \\ Y(i) & Y(j) \\}; \cofib{1-1}{1-2} \cofib{2-1}{2-2}
    \to{1-1}{2-1} \to{1-2}{2-2}
  \end{diagram}
  is good.
\end{definition}

As applying $S_\dot$ to a Waldhausen category produces a simplicial Waldhausen
category we can iterate the construction.  It is not very difficult to see that
the $k$-fold iterated construction $S^{(k)}_\dot\C$ has objects which are functors 
\[X\colon\Ar(\< n_1\>\timeses \< n_k\>) \rto \C,\] such that
for every $k$-cube $I:\I^k \rto \Ar\< \vec n \>$, the $k$-cube $X\circ
I$ is good, and has as cofibrations the natural transformations $\alpha:X\rto Y$
such that the $k+1$-cube $(X\circ I \Rto^\alpha Y\circ I)$ is good.

\begin{definition} \lbl{def:K}
  Let $\Sp$ be the category of symmetric spectra.  The functor $K:\WaldCat \rto
  \Sp$ is defined by taking a Waldhausen category $\C$ to the symmetric spectrum
  \[(|w\C|, |wS_\dot\C|, |wS^{(2)}_\dot\C|, \ldots).\]
\end{definition}

For more details, see \cite{waldhausen83}; for more on an all-at-once
construction of $K(\C)$, see \cite[Section 2]{blumbergmandell11}.

\begin{definition}
  A functor $F:\C\times\D \rto \E$ of Waldhausen categories is \textsl{biexact}
  if the following conditions hold.
  \begin{enumerate}
  \item For any object $A\in \C$, $F(A,-)$ is exact; for any object $B\in \D$,
    $F(-,B)$ is exact.
  \item For any two cofibrations $f:A \rcofib A'\in \C$ and $g:B\rcofib B'\in \D$,
    the southern arrow of the square
    \begin{diagram}[3.5em]
      {F(A,B) & F(A',B) \\ F(A,B') & F(A',B') \\};
      \arrowsquare{F(f,1_B)}{F(1_A,g)}{F(1_{A'},g)}{F(f,1_{B'})}
    \end{diagram}
    is a cofibration.
  \end{enumerate}
\end{definition}
The definition of biexact functor is meant to be analogous to the definition of
bilinear map.  If $\C$, $\D$ and $\E$ are all equal then this should correspond
to a product structure on $K(\C)$.  In an ideal situation, $\WaldCat$ would have
a monoidal structure $\otimes$ representing biexact functors, and all we would
need to show is that $K$ is symmetric monoidal.  Unfortunately, this cannot
happen:
\begin{proposition}
  There does not exist a symmetric monoidal product $\otimes$ on $\WaldCat$ such
  that 
  \[\left\{
    \begin{array}{c}
      \hbox{exact functors} \\ \C\otimes \D \rto \E
    \end{array}
  \right\} \bothto \left\{
    \begin{array}{c}
      \hbox{biexact functors} \\ \C\times \D \rto \E
    \end{array}
  \right\}.\]
\end{proposition}

This result is well-known to experts, but as we could not find a reference for
it we present a proof here.

\begin{proof}
  Let $\N_*$ be the full subcategory of $\FinSet_*$ with objects the finite
  pointed sets $\underbar{n} \defeq \{*,1,\ldots,n\}$ for $n\geq 0$; note that
  $\N_*$ is equivalent to $\FinSet_*$.  As all Waldhausen categories contain all
  binary coproducts, any Waldhausen category contains $\N_*$ as a Waldhausen
  subcategory.  Suppose that such a symmetric monoidal structure on $\WaldCat$
  exists, and let $\mathcal{S}$ be the unit.  Then by assumption,
  $\mathcal{S}\otimes \N_* \cong \N_*$ and the set of exact functors $\N_* \rto
  \FinSet_*$ is given by a choice of sets $(A_1,\ldots,A_n,\ldots)$ such that
  $|A_n| = n|A_1|$.

  Let $\iota: \N_* \rto \mathcal{S}$ be any inclusion of $\N_*$ as a subcategory
  of $\mathcal{S}$, and let $F:\mathcal{S}\times \N_* \rto \FinSet_*$ be any
  biexact functor. By assumption, $F$ must be uniquely determined by a choice of
  sets $(A_1,\ldots)$.  However, unlike in $\N_*$, in $\mathcal{S}\times \N_*$ there are
  multiple objects whose image under $F$ must have the same cardinality; for
  example,
  \[\big|F(\iota(\underbar{1}) \amalg \iota(\underbar{1}), \underbar{1})\big| =
  \big| F(\iota(\underbar{1}), \underbar{1}\amalg \underbar{1}) \big|.\]
  Thus a single set $A_n$ cannot determine the value of $F$, and we see that
  such a bijection cannot exist.
\end{proof}

We want to show that even though $\WaldCat$ does not have a symmetric monoidal
structure, it does have the next best thing: a symmetric multicategory structure
where the $2$-ary morphisms are exactly the biexact functors.

\section{Cubes}

In this section we develop some technical aspects of the theory of cubes.  The
category $\C$ will always be assumed to be a Waldhausen category.

The general idea of this section is that good $n$-cubes should behave like
objects in a Waldhausen category, and that cofibrations between them should
correspond to good $n+1$-cubes.  More formally, we have the following
proposition, which is designed to be a higher-dimensional analog of Axiom~(W4).

\begin{proposition} \lbl{prop:good-pushout}
  Let $I$, $J$ and $K$ be good $n$-cubes in $\C$, and suppose that $\alpha:I
  \Rto J$ is a natural transformation.  If $[\alpha]$ is a good $n+1$-cube then
  the diagram
  \[K \Lto I \Rto^\alpha J\] has a pushout $J\cup_IK$, and the natural
  transformation $\beta: K \Rto J\cup_I K$ gives a good $n+1$-cube $[\beta]$.
\end{proposition}

Note that as $J\cup_I K$ is a face of $[\beta]$ it must be good as well.

The proof of this proposition is quite long, so to aid understanding we first
spend some time developing a general theory of cubes.  The first result we
mention is the $n=1$ case of the proposition, which is proved as lemma 1.1.1 in
\cite{waldhausen83}:

\begin{lemma} \lbl{lem:case1}
  Consider the diagram
  \begin{diagram}
    {C & A & B \\ C'& A'& B'\\};
    \to{1-2}{1-1} \cofib{1-2}{1-3}
    \cofib{1-1}{2-1} \cofib{1-2}{2-2} \cofib{1-3}{2-3}
    \to{2-2}{2-1} \cofib{2-2}{2-3}
  \end{diagram}
  in $\C$.  If the right-hand square is good (as a $2$-cube) then the induced
  morphism $B\cup_A C \rcofib B'\cup_{A'}C'$ is a cofibration.
\end{lemma}


The next couple of lemmas are general category-theoretic observations whose
proofs are simple, but which we will need several times in the forthcoming
proofs.

\begin{lemma}  \lbl{lem:colim-comp}
  If $\D$ is any category with a terminal object $*$ and
  $F:\A\times \D \rto \C$ is a functor, then
  \[\colim F \cong \colim F\big|_{\A\times \{*\}}.\]
\end{lemma}

\begin{proof}
  This follows because the functor $\A\times\{*\} \rto \A\times\D$ is cofinal.
\end{proof}

Many of our proofs rely on computing southern arrows of cubes; luckily, it turns
out that these can be deduced from simple pushouts.  The following lemma is an
$n$-dimensional generalizataion of a special case of the rigid dual of
Proposition~0.2 in \cite{goodwilliecalcii}; we state it here as we will be using
it several times in this section.  We defer the proof until
Appendix~\ref{app:technical}.

\begin{lemma} \lbl{lem:pushout}
  Suppose that we are given an indexing category $\A$ along with $n$
  subcategories $\A_1,\ldots,\A_n$ such that $\A = \bigcup_{i=1}^n \A_i$.  Let
  $F:\A \rightarrow \C$.  Then the southern arrow of the cube $I$ given by
  \[I(\bar \epsilon) =
  \begin{cases}
    \colim F\big|_{\bigcap_{\epsilon_i= 0} \A_i} & \bar \epsilon \neq 
    (1,\ldots,1) \\ 
    \colim F & \bar \epsilon = (1,\ldots,1)
  \end{cases}\]
  is an isomorphism.
\end{lemma}

The special case that we will use the most often is the following.  Let $I:\I^n
\rto \C$ be an $n$-cube.  Set $\A = \I^n \backslash \{(1,\ldots,1)\}$, $\A_1 =
\A \backslash \{(1,\ldots,1,0)\}$, and $\A_2 = \I^{n-1}\times\{0\}$.  Then 
\begin{eqnarray*}
  \colim I|_{\A_1} &=& \colim (I_{n1})' \\
  \colim I|_{\A_2} &\cong& I(1,\ldots,1,0) \\ 
  \colim I|_{\A_1\cap \A_2} &=& \colim (I_{n0})'.
\end{eqnarray*}
Applying
Lemma~\ref{lem:pushout} we have a pushout square \label{eq:pushout}
\begin{diagram} 
  {\colim {}(I_{n0})' & I(1,\ldots,1,0) \\ \colim {}(I_{n1})' & \colim I'\\};
  \arrowsquare{}{}{}{}
\end{diagram}
which will allow us to compute the southern arrow of $I$ using pushouts and
induction. 

We now turn to the existence of southern arrows.

\begin{lemma} \lbl{lem:semi-existence}  \lbl{lem:induced-cofib}
  Let $I$ be a good $n$-cube in a Waldhausen
  category $\C$.  Then $\colim I_{n0}' \rto \colim I_{n1}'$ is a cofibration,
  and the southern arrow of $I$ exists.
\end{lemma}

\begin{proof}
  We prove that $\colim I_{n0}' \rto \colim I_{n1}'$ is a cofibration by
  induction on $n$.  The cases $n=1,2$ follow directly from the definition of a
  good $n$-cube; the case $n=3$ is a special case of Lemma~\ref{lem:case1}.  Now
  suppose that this is true for all cubes of size less than $n$, and consider
  the situation for $n$-cubes.

  Let $A = I(0,\ldots,0)$, $B = I(0,\ldots,0,1)$, $X = I(1,\ldots,1,0,0)$ and $Y
  = I(1,\ldots,1,0,1)$.  Then (by Lemma~\ref{lem:pushout}) we know that 
  \[\colim I_{n0}' \cong\colim\Big(\colim ((I_{n0})_{(n-1)1})'
  \lto A \rcofib X\Big)\]
  and
  \[\colim I_{n1}' \cong\colim\Big((\colim (I_{n1})_{(n-1)1})'
  \lto B \rcofib Y\Big).\]
  Note that $(I_{n0})_{(n-1)1} = (I_{(n-1)1})_{(n-1)0}$ and $(I_{n1})_{(n-1)1} =
  (I_{(n-1)1})_{(n-1)1}$, so the inductive hypothesis applies to these and we
  get a diagram
  \begin{diagram}
    { \colim ((I_{n0})_{(n-1)1})' & A & X \\
      \colim ((I_{n1})_{(n-1)1})' & B & Y \\};
    \cofib{1-1}{2-1} \cofib{1-2}{2-2} \cofib{1-3}{2-3}
    \to{1-2.mid west}{1-1.mid east} \to{2-2.mid west}{2-1.mid east}
    \cofib{1-2.mid east}{1-3.mid west} \cofib{2-2.mid east}{2-3.mid west}
  \end{diagram}
  As $I$ is good the right-hand square is also good, and thus we see that
  Lemma~\ref{lem:case1} applies and the induced morphism between the pushouts is
  a cofibration, as desired.

  In order for the southern arrow to exist we need to show that $\colim I'$
  exists.  By Lemma~\ref{lem:pushout} we know that
  \begin{eqnarray*}
    \colim I' &\cong& \colim \big( \colim I_{n0} \lto \colim I'_{n0} \rto \colim
        I\big|_{\bar\epsilon \notin \{1\}^{n-1}\times\I}\big) \\
        &\cong& \colim \big(I_{n0}(1) \lto \colim I_{n0}'\rto \colim I_{n1}'\big),
  \end{eqnarray*}
  where $\colim I\big|_{\bar\epsilon \notin \{1\}^{n-1}\times\I} \cong \colim
  I'_{n1}$ by Lemma~\ref{lem:colim-comp}.  But by the first part of the proof
  the second morphism in the colimit is a cofibration, so the pushout exists.
\end{proof}

As a consequence we get the following:

\begin{lemma}  \label{lem:semi-existence2}
  In a pushout square of $n$-cubes 
  \begin{diagram}
    { I & J \\ K & J\cup_I K \\};
    \arrowsquare{\alpha}{}{}{\beta}
  \end{diagram}
  where $I$, $J$, $K$ and $[\alpha]$ are good, the southern arrows of
  $J\cup_IK$ and $[\beta]$ exist.
\end{lemma}

\begin{proof}
  As colimits commute we know that
  \[\colim {} (J\cup_I K)' \cong \colim \Big(\colim K' \lto
  \colim I' \rcofib^i \colim J'\Big),\]
  where $i$ is a cofibration by Lemma~\ref{lem:induced-cofib} because
  $[\alpha]$ is good.  Each term in the pushout on the right exists because $I$,
  $J$ and $K$ are good, and the pushout itself exists because $i$ is a
  cofibration; thus the southern arrow of $J\cup_I K$ exists.

  As $K$ is good we know that its southern arrow is a cofibration, and by the
  special case of Lemma~\ref{lem:pushout} discussed on page~\pageref{eq:pushout}
  the southern arrow of $[\beta]$ exists if the southern arrow of
  $[\beta]_{(n+1)1}$ exists.  But $[\beta]_{(n+1)1}$ is exactly $J\cup_I K$
  which has a southern arrow, as desired.
\end{proof}



Now we are ready to prove Proposition~\ref{prop:good-pushout}.

\begin{proof}[Proof of Proposition~\ref{prop:good-pushout}]
  Note that $I$ gives us a natural transformation of functors from $I'$ to the
  constant functor at $I(1)$.
  
  We will prove the desired statement by induction on $n$.  Clearly for $n=0$
  we just need to show that cofibrations are preserved under pushout
  in $\C$, which we know is true because $\C$ is a Waldhausen category.  For
  $n=1$ this is just Lemma~\ref{lem:case1}.  Thus assume that $n>1$ and that we
  know that the proposition holds for all smaller dimensions.  By restricting to
  faces of $I$,$J$,$K$ and using the inductive hypothesis we see that any face
  of $[\beta]$ that is not equal to $K$ or $J\cup_I K$ is good.  $K$ is given to
  be good, so in order to prove the proposition it suffices to show that the
  southern arrows of $J\cup_I K$ and $[\beta]$ are cofibrations; we know that
  they exist by Lemmas~\ref{lem:semi-existence} and \ref{lem:semi-existence2}.
  
  The southern arrow of $\beta$ is the morphism
  \[j:\colim\Big(\makeshort{(J\cup_IK)' \Lto^\beta K'\Rto K(1)}\Big)  \rto (J\cup_IK)(1);\]
  we want to show that it is a cofibration.  Note that
  \[\begin{array}{l}
    \colim\begin{inline-diagram} {K' & K(1) \\ (J\cup_IK)'
        \\}; \To{1-1.mid east}{1-2.mid west} \To{1-1}{2-1}_\beta \end{inline-diagram} \cong 
    \colim \begin{inline-diagram} {I' & K ' & K(1) \\
        J' &  (J\cup_I K) ' \\}; \To{1-1.mid east}{1-2.mid west}
      \To{1-2.mid east}{1-3.mid west} \To{1-1}{2-1}_\alpha \To{1-2}{2-2}^\beta
      \To{2-1.mid east}{2-2.mid west}\end{inline-diagram} \\
    \qquad\cong \colim \big( J'\Lto^\alpha I' \Rto K(1) \big)
    \\ \qquad \cong \colim
    \big( J'\Lto^\alpha I'\Rto I(1) \rto K(1) \big)
  \end{array}\]
  where the second line follows because the square on the first line is a
  pushout square.  Therefore we have the folowing diagram, where the three
  commutative squares are all pushout squares:
  \begin{diagram}
    {I' & I(1) & K(1) \\
      J ' & \colim {} [\alpha]'  & K(1)\cup_{I(1)} \colim {} [\alpha]
\\
       & J(1) & (J\cup_IK)(1) \\};
    \To{1-1.mid east}{1-2.mid west} \to{1-2.mid east}{1-3.mid west} \to{1-2}{2-2}
    \To{1-1}{2-1}_\alpha \To{2-1.mid east}{2-2.mid west}
    \cofib{2-2}{3-2}^i \To{2-1}{3-2} \to{2-2.mid east}{2-3.mid west}
    \to{3-2.mid east}{3-3.mid west} \to{1-3}{2-3} \cofib{2-3}{3-3}^j
  \end{diagram}
  Since $[\alpha]$ is a good cube it follows that $i$ is a cofibration, and thus
  $j$ is a cofibration because it is a pushout of $i$.

  It now remains to show that the southern arrow of $J\cup_I K$ is a
  cofibration.  The southern arrow of $J\cup_I K$ factors through the
  southern arrow of $[\beta]$; thus it suffices to show that the connecting
  morphism is also a cofibration.  By Lemma~\ref{lem:pushout} we know that
  \begin{eqnarray*}
    \colim [\beta]'&\cong& \colim \big( \colim K \lto \colim K'\rto \colim {}
    [\beta]\big|_{\bar\epsilon \notin \{1\}^k\times \I} \big) \\
    &\cong& \colim \big( K(1) \lcofib \colim K' \rto \colim {} (J\cup_I K)'\big)
  \end{eqnarray*}
  where the left-hand map on the bottom is a cofibration because $K$ is a good
  cube, and $\colim[\beta]\big|_{\bar\epsilon \notin \{1\}^n\times\I} \cong
  \colim (J\cup_IK)'$ by Lemma~\ref{lem:colim-comp}.  The induced morphism
  $\colim (J\cup_IK)' \rto \colim [\beta]'$ is the right-hand map in the given
  pushout square.  Given that it is the pushout of the cofibration $\colim K'
  \rcofib K(1)$ we know that it is also a cofibration, as desired.  Thus $J\cup_I
  K$ is a good cube, and we are done.
\end{proof}

\section{$k$-exactness}

The goal of this section is to show that $\WaldCat$ is a symmetric
multicategory, where $\WaldCat(\C,\D)$ is the set of exact functors $\C \rto \D$
and $\WaldCat_k(\C_1,\ldots,\C_k;\D)$ is the set of $k$-exact functors.  Most of
this section is devoted to defining $k$-exactness and working out its
properties.

When $k$ is clear from context, we will write $\bar\C = \C_1\timeses \C_k$, and
refer to objects $\bar A = (A_1,\ldots,A_k)\in\bar \C$, and
$\bar f = (f_1,\ldots,f_k)\in \bar\C$.

\begin{definition}
  Given morphisms $f_i:A_{i0}\rto A_{i1}\in \C_i$ and a functor $F:\bar \C \rto
  \D$ we define a $k$-cube $[\bar f]_F$ in $\D$ by 
  \[[\bar f]_F(\epsilon_1,\ldots,\epsilon_k) =
  F(A_{1\epsilon_1},\ldots,A_{k\epsilon_k})\] 
  and 
  \[[\bar f]_F(\makeshort{\epsilon_1\rto \epsilon_1',\ldots,\epsilon_k\rto
    \epsilon_k'}) = F(h_1,\ldots,h_k),\] where $h_i = 1_{A_{i\epsilon_i}}$ if
  $\epsilon_i = \epsilon_i'$ and $h_i = f_i$ otherwise.  We define the
  \textsl{box product} of $\bar f$ to be the southern arrow of $[\bar f]_F$.  
\end{definition}

\begin{definition}
  A $0$-exact functor with codomain $\D$ is an object of $\D$.  A functor
  $F:\C_1\timeses \C_k \rto \D$ is \textsl{$k$-exact} if
  \begin{itemize}
  \item[(kE1)] $F(\bar A) = 0$ if $A_i = 0$ for any $1\leq i \leq k$,
  \item[(kE2)] $F$ preserves pushsouts in each variable independently,
  \item[(kE3)] $F(w\bar\C) \subseteq w\D$, and 
  \item[(kE4)] for all $\bar f\in c\bar \C$, $[\bar f]_F$ is good.
  \end{itemize}

  If a functor is $k$-exact for some $k$, we will call it \textsl{multiexact}.
\end{definition}

\begin{definition}
  Let $\WaldCat$ be the following symmetric multicategory:
  \begin{description}
  \item[objects] Waldhausen categories.
  \item[$k$-morphisms] $k$-exact functors $\C_1\timeses \C_k \rto \D$.
  \item[$\Sigma_k$-action] permuting the input variables.
  \end{description}
\end{definition}

To check that this is well-defined it suffices to check that multiexact functors
compose properly.  This is exactly the result of
Proposition~\ref{lem:k-composition}; however, before we can prove that we need
to develop a little bit of the theory of $k$-exact functors.  The first lemma we
prove shows that in order to prove Axiom~(kE4) it suffices to consider a single
type of southern arrow.

\begin{lemma} \lbl{lem:restricted-3} Let $F:\C_1\timeses \C_k \rto \D$ be any
  functor satisfying Axiom~(kE1).  If the southern arrow of $[\bar f]_F$ is a
  cofibration for all $\bar f\in c\bar\C$ then $F$ satisfies Axiom~(kE4).
\end{lemma}

\begin{proof}
  Fix $\bar f\in c\C$, writing $f_i:A_{i0}\rto A_{i1}$.  We know that the
  southern arrow of $[\bar f]_F$ is a cofibration, so all that we need to show
  is that all faces of $[\bar f]_F$ are good.  Let $J\subseteq \{1,\ldots,n\}$
  and let $\bar\eta\in\{0,1\}^k$.  A face of $[\bar f]_F$ is determined by
  $J$ and $\bar \eta$ by considering the cube $I_{J,\bar\eta}:\I^{|J|}
  \rto \D$ given by
  \[I_{J,\bar\eta}(\epsilon_{j_1},\ldots,\epsilon_{j_{|J|}}) = F(B_1,\ldots,B_k)
  \qquad B_j = \begin{cases} A_{j\eta_j} & \hbox{if } j\notin J \\
    A_{j\epsilon_j} & \hbox{if }j\in J.\end{cases}\] More informally, we let $J$
  determine which variables are allowed to change, and let $\bar\eta$ determine
  the value of the other variables.  Let $\bar h$ have $h_j = f_j$ if $j\in J$
  and $0 \rto A_{j\eta_j}$ otherwise.  Then $\bar h \in c\bar \C$, so by
  assumption we know that the southern arrow of $[\bar h]_F$ is a cofibration.
  But this southern arrow is exactly the southern arrow of
  $I_{J,\bar\eta}$, so we see that the southern arrow of
  $I_{J,\bar\eta}$ is a cofibration, as desired.
\end{proof}

We now need to consider composition of multiexact functors.  First, a simple
observation about southern arrows.

\begin{lemma} \lbl{lem:codiag-1var}
  A $k$-exact functor commutes with taking southern arrows in each variable
  independently. 
\end{lemma}
\begin{proof}
  This follows from Axiom~(kE2) and Lemma~\ref{lem:pushout}, which says that
  southern arrows can be computed as iterated pushouts.
\end{proof}

The next two results concern the following situtation.  Let $j_1,\ldots,j_k\in
\Z_{\geq 0}$, and write $m_i = \sum_{\ell=1}^i j_i$.  Given $j_i$-exact functors
\[G_i:\C_{m_{i-1} + 1} \timeses \C_{m_i} \rto \D_i\] and a $k$-exact functor
\[F:\D_1\timeses \D_k \rto \D,\] define the composite functor
\[H=F\circ (G_1\timeses G_k): \C_1\timeses \C_{m_k} \rto \D.\]

\begin{lemma} \lbl{lem:composition-good} Let $g_i$ be the southern arrow of
  $[(f_{m_{i-1}+1},\ldots,f_{m_i})]_{G_i}$. The southern arrow of $[\bar f]_H$
  is isomorphic to the southern arrow of $[\bar g]_F$.
\end{lemma}

\begin{proof}
  We will use Lemma~\ref{lem:pushout}.  
  Let $\A_i = \I^{m_{i-1}} \times (\I^{j_i}\backslash\{(1,\ldots,1)\} \times \I^{m_k-m_i}$, so that
  for any $J\subseteq \{1,\ldots,k\}$,
  \[\bigcap_{j\in J} \A_j = \prod_{i=1}^k B_i^{(J)}, \qquad\hbox{where
  }B_i^{(J)} =  \begin{cases} \I^{j_i} & \hbox{if
    }i\in J \\ \I^{j_i} \backslash \{(1,\ldots,1)\} &  \hbox{if } i\notin J\end{cases}\] 
  and
  \[\bigcup_{i=1}^n \A_i = \I^{m_k}\backslash\{(1,\ldots,1)\}.\]
  In addition, by Lemma~\ref{lem:codiag-1var}
  \[\colim F\circ (G_1\timeses G_k)\big|_{\bigcap_{j\in J} \A_j} \cong
  F\left(\colim G_1 \big|_{B_1^{(J)}}, \ldots, \colim G_k \big|_{B_k^{(J)}}
  \right).\] Note that these are exactly the entries of the cube $[\bar g]_F$.
  Applying Lemma~\ref{lem:colim-comp} to compute $\colim {}[\bar f]_H'$ in terms
  of these, we see that the southern arrow of $[\bar f]_H$ is exactly the
  southern arrow of $[\bar g]_F$, as desired.
\end{proof}

Using this we can show that $H$ is $m_k$-exact, and thus that $\WaldCat$ is a
multicategory. 

\begin{proposition} \lbl{lem:k-composition} $\WaldCat$ is a multicategory.
\end{proposition}

\begin{proof}
  It suffices to check that $H$ is $m_k$-exact.  Axioms~(kE1), (kE2) and (kE3) are
  direct from the definition, so we just need to check Axiom~(kE4).
  
  Let $\bar f\in \bar c\C$; we want to show that $[\bar f]_H$ is good.  By
  Lemma~\ref{lem:restricted-3} it suffices to show that the southern arrow of
  $[\bar f]_H$ is a cofibration.  Let $g_i$ be the southern arrow of
  $[(f_{m_{i-1}+1},\ldots,f_{m_i})]_{G_i}$.  As $G_i$ is $j_i$-exact we know
  that $g_i$ is a cofibration; as $F$ is $k$-exact, the southern arrow of
  $[\bar g]_F$ is also a cofibration.  However, by
  Lemma~\ref{lem:composition-good} we know that the southern arrow of $[\bar
  g]_F$ is exactly the southern arrow of $[\bar f]_H$, so we see that the
  southern arrow of $[\bar f]_H$ must also be a cofibration, as desired.
\end{proof}

\section{The closed structure}

We would now like to show that $\WaldCat$ is a closed multicategory.  

\begin{definition}
  We define the internal hom $\uWaldCat(\C_1,\ldots,\C_k;\D)$ in the following
  manner: 
  \begin{description}
  \item[objects] $k$-exact functors $\C_1\timeses \C_k \rto \D$,
  \item[morphisms] natural transformations between functors,
  \item[weak equivalences] natural weak equivalences between functors, and 
  \item[cofibrations] natural transformations $\alpha:F\rto G$ such that for any
    $\bar f \in c\bar \C$, the cube \[\left([\bar f]_F \Rto^\alpha [\bar f]_G\right)\] is good.
  \end{description}
\end{definition}

In particular, note that all cofibrations are levelwise cofibrations.

We need to prove that this is well-defined.

\begin{lemma}
  $\uWaldCat(\C_1,\ldots,\C_k;\D)$ is a Waldhausen category.
\end{lemma}

\begin{proof}
  The only axiom that is nontrivial to check is Axiom~(W4); the others follow
  directly from the fact that $\D$ is a Waldhausen category.  Thus we focus our
  attention on checking Axiom~(W4).

  Consider a diagram
  \[H \lto F \rcofib^\alpha G \in \uWaldCat(\bar\C;\D).\]
  We know that $\alpha$ is a levelwise cofibration and that pushouts along
  cofibrations exist in $\D$, so we get a functor $G\cup_F H:\bar\C\rto \D$.  We
  need to check that this functor is $k$-exact, and that the induced natural
  transformation $\beta:H\rto G\cup_FH$ is a cofibration inside
  $\uWaldCat(\bar\C;\D)$.

  The functor $G\cup_F H$ satisfies Axiom~(kE1) and (kE2) because $F$, $G$ and $H$
  are $k$-exact, and has Axiom~(kE3) because Axiom~(W5) holds in $\D$.  To prove
  Axiom~(kE4), fix $\bar f\in c\bar \C$.  We know that $[\bar f]_F$, $[\bar
  f]_G$ and $[\bar f]_H$ are good, so $[\bar f]_{G\cup_F H}$ is also good by
  Proposition~\ref{prop:good-pushout}.  Thus $G\cup_F H$ is $k$-exact.
  Proposition~\ref{prop:good-pushout} also tells us that $([\bar f]_H \Rto^\beta
  [\bar f]_{G\cup_F H})$ is good, which means that $\beta$ is a cofibration, as
  desired.
\end{proof}

In order for this definition to make $\WaldCat$ into a closed multicategory, we
need a $k+1$-exact evaluation morphism.  

\begin{definition} \lbl{def:evaluation} 
  The functor \[\ev_{\C_1,\ldots,\C_k;\D}: \C_1\times\cdots\times
  \C_k \times \uWaldCat(\C_1,\ldots,\C_k;\D) \rto \D \] is defined by 
  \[\ev_{\C_1,\ldots,\C_k;\D}(A_1,\ldots,A_k,F) = F(A_1,\ldots,A_k).\]
\end{definition}

\begin{lemma} 
  The functor $\ev_{\C_1,\ldots,\C_k;\D}$ is $k+1$-exact.
\end{lemma}

\begin{proof}
  Axioms~(kE1), (kE2) and (kE3) are easily checked from the fact that $F$ is
  $k$-exact.  Thus we only need to check Axiom~(kE4). Given $f_i:A_i\rcofib
  B_i\in \C_i$ and a cofibration $\alpha:F\rto G$ we need to check that $[(\bar
  f, \alpha)]_\ev$ is good.  However, by definition this cube is the cube $([\bar
  f]_F \Rto^\alpha [\bar f]_G)$, which is good because $F$ and $G$ are $k$-exact
  and $\alpha$ is a cofibration.
\end{proof}

It remains to check that this gives a well-defined closed multicategory
structure on $\WaldCat$, i.e. that for all Waldhausen categories
$\C_1,\ldots,\C_k$, $\D$ and $\A_1,\ldots,\A_\ell$ the function
\[\begin{array}{l} \WaldCat_\ell(\A_1,\ldots,\A_\ell; \uWaldCat(\C_1,\ldots,\C_k;\D)) \\
  \rto
  \WaldCat_{k+\ell}(\C_1,\ldots,\C_k,\A_1,\ldots,\A_\ell; \D)\end{array}\] given by $F
\rgoesto \ev_{\C_1,\ldots,\C_k,\D} \circ (1_{\C_1},\ldots,1_{\C_k},F)$ is a
bijection.  We can construct an inverse easily by partial application, so
assuming that the partial inverse is well-defined we know that this is a
bijection.  It is also clearly $\Sigma_k\times \Sigma_\ell$-equivariant.  Thus we need to show the following:
\begin{lemma}
  Fix $1\leq \ell \leq k$.  Let $F:\C_1\times\cdots\times \C_k \rto \D$ be a
  $k$-exact functor, and fix $A_i \in \C_i$ for $\ell < i\leq k$.  Then the
  functor 
  \[F(-,A_{\ell+1},\ldots,A_k):\C_1\times\cdots\times \C_\ell \rto \D\]
  is $\ell$-exact.
\end{lemma}

\begin{proof}
  It suffices to prove this for $\ell=k-1$; the rest will follow by induction.
  Axioms (kE1), (kE2) and (kE3) hold immediately, so it suffices to consider
  (kE4).  Let $f_i \in c\C_i$ for $1\leq i < k$ be cofibrations, and let
  $f_k:\initial \rcofib A_k$ for the $A_k$ in the statement of the lemma.  Since
  $F$ was $k$-exact, we know that $[\bar f]_F$, which implies that $([\bar
  f]_F)_{k1}$ is good.  But this is exactly the cube
  $[(f_2,\ldots,f_k)]_{F(-,A_k)}$.  Thus $F(-,A_k)$ is $k-1$-exact, as desired.
\end{proof}

We have now proved:
\begin{proposition}
  $\WaldCat$ is a closed multicategory.
\end{proposition}

\section{$K$-Theory as an enriched multifunctor} \lbl{sec:KtheoryWald}

Our goal for this section is to show that the closed multicategory structure
on $\WaldCat$ is compatible with the $K$-theory functor.  

\begin{definition}
  A \textsl{multifunctor} $F:\M\rto \M'$ between symmetric multicategories is a
  function $F:\ob \M \rto \ob \M'$, and a function 
  \[\M(A_1,\ldots,A_k;B) \rto
  \M'(F(A_1),\ldots,F(A_k); F(B))\] for all tuples $B,A_1,\ldots,A_k$ of objects.
  These must preserve the units and be compatible with composition and the
  $\Sigma_k$-action.  In the case when $\M$ and $\M'$ are enriched over a
  symmetric monoidal $\V$, we
  just need 
  \[\M(A_1,\ldots,A_k;B) \rto \M'(F(A_1),\ldots,F(A_k); F(B))\] to be a
  $\V$-morphism. 
\end{definition}

First we consider the unenriched setting.  

\begin{proposition} \lbl{prop:Kmulti} The functor $K\colon \WaldCat\rto \Sp$ is
  a multifunctor.
\end{proposition}

This statement is well-known to specialists, and has been mentioned in many
papers, including \cite{waldhausen83}, \cite[Theorem~2.6]{blumbergmandell11}, and
\cite{geisserhesselholt99}.  However, we could not find a reference that
explicitly checked that the multifunctor structure was compatible with the
structure maps of the symmetric spectra produced by $K$-theory, so in the
interest of completeness we present the proof here, as well.

Before starting the proof of this proposition we will first make an auxillary
construction.  In order to show that $K$ is a multifunctor we will need to show
that any $k$-exact functor $F\colon \C_1\timeses \C_k \rto \C$ gives rise to a
morphism $K(\C_1)\smashes K(\C_k) \rto K(\C)$.  When $k=0$ this says that a
choice of object $A\in \C$ gives a morphism $\S \rto K(\C)$.  To construct this,
choose an equivalence $\alpha:\S \rto K(\FinSet_*)$, and then take the exact
functor $p_A:\FinSet_* \rto \C$ given by $I \rgoesto \coprod_I A$.  Then
$K(p_A)\alpha$ is the desired morphism.  In the case $k=1$ this is just a
definition check to see that the definition of $K$-theory in
Definition~\ref{def:K} is functorial in $\C$.

Now consider $k>1$.  In the interest of simplifying the following analysis, we
will restrict our attention to the case when $k=2$; the higher cases follow
analogously. The data of a $2$-morphism is, for every pair $m_1,m_2$, a map
of spaces
\[\mu_{m_1,m_2}\colon K(\C_1)_{m_1}\smash K(\C_2)_{m_2} \longrto
K(\C)_{m_1+m_2}.\] These maps need to be coherent with respect to the spectral
structure maps; in particular, we need the following diagram to commute:
\begin{diagram}[-.7em]
  { K(\C_1)_{m_1}\smash K(\C_2)_{m_2} \smash S^1 & & K(\C_1)_{m_1} \smash S^1 \smash K(\C_2)_{m_2} \\
     & K(\C_1)_{m_1} \smash K(\C_2)_{m_2+1} & K(\C_1)_{m_1+1} \smash K(\C_2)_{m_2} \\
    K(\C)_{m_1+m_2} \smash S^1 & K(\C)_{m_1+m_2+1} & K(\C)_{m_1+1+m_2} \\}; \to{1-1}{1-3}
  \to{1-1}{3-1}_{\mu_{m_1,m_2}} \to{1-1}{2-2} \to{1-3}{2-3} \to{3-1}{3-2}
  \to{2-3}{3-3}^{\mu_{m_1+1,m_2}} \to{3-3}{3-2} \to{2-2}{3-2}^{\mu_{m_1,m_2+1}}
\end{diagram}

For a Waldhausen category $\C$ and $0\leq i \leq n$ we define a functor
$\rho_{ni}\colon \C \rto S_n \C$, which is defined on objects by
\[\rho_{ni}(A)_{jk} = \begin{cases} * \caseif j \leq n-i \hbox{ or } k\geq i, \\ A \caseotherwise \end{cases}\]
and extends in the analogous manner to morphisms.  Let $S^1$ be the pointed
simplicial set which at level $n$ is equal to the set $\{0,1,\ldots,n\}$; we can
also consider $S^1$ to be a pointed category with only trivial morphisms.  Then
we have a morphism of simplicial categories $P\colon \C \times S^1 \rto S_\dot
\C$ $(A,i) \rgoesto \rho_{ni}(A).$

\begin{lemma}
$P$ is a well-defined functor of simplicial categories.
\end{lemma}

\begin{proof}
  In order for $P$ to be well-defined we need to show that the image of $P$ is
  in $S_\dot \C$, and that $P$ is compatible with the simplicial maps.  The first
  part of this is true by definition, since $\rho_{ni}$ is constructed to be a
  valid element of $S_n\C$.  For the second part, note that we have
  \begin{align*}
    \partial_j \rho_{ni}(A) &= \begin{cases} \rho_{(n-1)0}(A) \caseif j=0 \hbox{
        and } i=n \hbox{ or } j=n\hbox{ and } i=1,\\\rho_{(n-1)i}(A) \caseif j
      \leq n-i\hbox{ and } i\neq n \\ \rho_{(n-1)(i-1)}(A) \caseif j >
      n-i \end{cases}\\ &= \rho_{(n-1)(\partial_j(i))}(A),
  \end{align*}
where in the right-hand side of the above, $i\in S^1_n$.  Analogously,
\[s_j \rho_{ni}(A) = \begin{cases} \rho_{(n+1)i} \caseif j \leq n-i \\ \rho_{(n+1)(i+1)} \caseif j > n-i\end{cases} = \rho_{(n-1)(s_j(i))}(A), \]
so we are done.
\end{proof}

We thus have functors
\[ P \colon S^{(m)}_\dot \C \times S^1 \rto S^{(m+1)}_\dot \C.\] By
definition, if either $i=0$ or $A = *$ then $P(A,i) = *$, so $P$ lifts to a map
\[P \colon NwS^{(m)}_\dot \C \smash S^1 \rto NwS^{(m+1)}_\dot \C.\] This is the
spectral structure map of the $K$-theory of a symmetric spectrum.

\begin{proof}[Proof of Proposition~\ref{prop:Kmulti}] 
  Consider a biexact functor $F\colon \C_1\times \C_2\rto \C$.  We want to use
  $F$ to construct morphisms $\mu_{m_1,m_2}\colon K(\C_1)_{m_1}\smash
  K(\C_2)_{m_2}\rto K(\C)_{m_1+m_2}$.  

  The key fact we need about the objects of $S_{n_1}\cdots S_{n_m} \C$ is that
  they will be preserved by biexact functors in the following manner.  Consider
  the composition
  \[\begin{array}{l}
    S^{(m_1)}_{\vec n_1}\C_1 \times S^{(m_2)}_{\vec n_2} \C_2 \rto \big[\Ar [\vec n_1], \C_1\big] \times \big[\Ar [\vec n_2], \C_2\big] \rto \\
    \qquad \big[\Ar([\vec n_1]\times [\vec n_2]), \C_2\times \C_2\big] \rto^{F\circ} \big [\Ar ([\vec n_1]\times [\vec n_2]), \C\big].
  \end{array}\]
  The key extra condition on the objects of $S^{(m_1+m_2)}_{\vec n_1\vec n_2}\C$ is that this functor lands in $S^{(m_1+m_2)}_{\vec n_1\vec n_2}\C$.  By varying the coordinates of $\vec n_1$ and $\vec n_2$ these assemble into exact functors
  \[ S^{(m_1)}_\dot \C_1 \times S^{(m_2)}_\dot \C_2 \rto S^{(m_1+m_2)}_\dot
  \C.\] Applying $|Nw\cdot|$ to these and noting that any point with the
  basepoint as one of the coordinates gets mapped to the basepoint, we get maps
  \[ \mu_{m_1,m_2} \colon K(\C_1)_{m_1} \smash K(\C_2)_{m_2} \rto
  K(\C)_{m_1+m_2}.\]

  In order to check these assemble into a map $K(\C_1)\smash K(\C_2) \rto K(\C)$
  we that these satisfy the coherence conditions stated earlier.  In order to
  show this, we will show that the following diagram commutes:
  \begin{diagram}[0em]
    {S_\dot^{(m_1)}\C_1 \times S_\dot^{(m_2)} \C_2 \times S^1 &  & S_\dot^{(m_1)}\C_1 \times S^1 \times S_\dot^{(m_2)} \C_2 \\
       & S_\dot^{(m_1)}\C_1\times S_\dot^{(m_2+1)} \C_2 & S_\dot^{(m_1+1)}\C_1 \times S_\dot^{(m_2)} \C_2 \\
      S_\dot^{(m_1+m_2)}\C \times S^1 & S_\dot^{(m_1+m_2+1)}\C & S_\dot^{(m_1+1+m_2)} \C \\}; \to{1-1}{1-3}
    \to{1-1}{3-1}_F \to{1-1}{2-2}^{1\times P} \to{1-3}{2-3}^{P\times 1}
    \to{3-1}{3-2}_P \to{2-3}{3-3}^F \to{3-3}{3-2} \to{2-2}{3-2}^F
  \end{diagram}
  In fact, all of the morphisms except for the two horizontal morphisms are
  obtained by postcomposing functors $\Ar ([\vec n_1] \times [\vec n_2]) \rto
  s\Cat$ with $P$ or $F$.  The horizontal morphisms, on the other hand, permute
  both source and target categories, and then permute the source categories
  back; everything in between is, once again, postcomposing with $P$ or $F$.
  Thus in order for this diagram to commute it suffices to show that the diagram
  \begin{diagram}
    {\C_1 \times \C_2 \times S^1 & & \C_1\times S^1 \times \C_2\\
      & \C_1\times S_\dot\C_2 & S_\dot\C_1\times \C_2 \\
      \C\times S^1  & S_\dot\C\\}; \to{1-1}{3-1}_{F\times 1} \to{1-1}{2-2}^{1\times
      P} \to{1-1}{1-3} \to{1-3}{2-3}^{P\times 1} \to{3-1}{3-2}_P
    \to{2-2}{3-2}^F \to{2-3}{3-2}^F
  \end{diagram}
  commutes.  Consider a triple $(A_1,A_2,i)\in W_1\times W_2\times S^1_n$.  To
  check that the diagram commutes, we need to show that
  \[\rho_{ni}(F(A_1,A_2)) = F(A_1,\rho_{ni}(A_2)) = F(\rho_{ni}(A_1), A_2).\]
  Looking at each of these at spot $jk$ we have that if $j\leq n-i$ or $k\geq
  i$, the first is $*$, the second is $F(A_1,*)$ and the third is $F(*,A_2)$,
  which are all equal because $F$ is biexact.  Otherwise, these are all equal to
  $F(A_1,A_2)$, so are again all equal.  So these diagrams commute on objects.
  Analogously, they commute on all morphisms.

  This completes the proof of Proposition~\ref{prop:Kmulti}.
\end{proof}

However, we still have not shown that $K$ is compatible with the closed
structure on $\WaldCat$.  However, as we have just shown that $K$ is in fact a
multifunctor, we can consider $\WaldCat$ to be enriched over $\Sp$ just by
applying $K$ to each of the internal hom-objects to produce a
spectrally-enriched category $\WaldCat_\Sp$.  We then have the following:

\begin{proposition}
  $K:\WaldCat_\Sp \rto \Sp$ is a spectrally-enriched multifunctor.
\end{proposition}

\begin{proof}
  We need to show that $K$ gives a morphism
  \[K(\uWaldCat(\C_1,\ldots,\C_k;\D)) \rto
  K(\D)^{K(\C_1)\smashes K(\C_k)}.\]  As $\Sp$ is closed symmetric monoidal, it suffices
  to show that we get a morphism $K(\C_1)\smashes K(\C_k) \smash
  K(\uWaldCat(\C_1,\ldots,\C_k;\D)) \rto K(\D)$.  The evaluation functor defined in
  Definition~\ref{def:evaluation} gives us such a morphism, and it follows
  directly that this produces the enrichment on $K$.
\end{proof}

For more on multifunctors between closed multicategories see
\cite[Section~3]{manzyuk12}.

\appendix

\section{Proof of Lemma~\ref{lem:pushout}} \lbl{app:technical}

For ease of reading we restate Lemma~\ref{lem:pushout} here. 

\begin{lemmapushout} 
  Suppose that we are given an indexing category $\A$ along with $n$
  subcategories $\A_1,\ldots,\A_n$ such that $\A = \bigcup_{i=1}^n \A_i$.  Let
  $F:\A \rightarrow \C$.  Then the southern arrow of the cube $I$ given by
  \[I(\bar \epsilon) = \colim F\big|_{\bigcap_{\epsilon_i= 0}
    \A_i}\qquad I(1) = \colim F\]
  is an isomorphism.
\end{lemmapushout}

\begin{proof}
  We want to show that $\colim I' \cong \colim F$; we will prove this by
  induction on $n$.  

  We begin by proving the inductive step.  Suppose that we know that this lemma
  is true for values lower than $n$.  Let $\A' = \bigcup_{i=1}^{n-1} \A_i$.  Let
  $\B = \I^{n} \backslash \{(0,1,\ldots,1),(1,\ldots,1)\}$.  By the $n=2$ case
  of the lemma applied to $I'$,
  \begin{diagram}
    {\colim I'\big|_{\B \cap (\I^{n-1}\times\{0\})} & \colim
      I'\big|_{\I^{n-1}\times \{0\}} \\
      \colim I'\big|_{\B} & \colim I' \\};
    \arrowsquare{}{}{}{}
  \end{diagram}
  is a pushout square.  By the $n-1$ case of the lemma, the upper-left corner is
  $\colim F\big|_{\A'\cap \A_n}$, and the lower-left corner is $\colim
  F\big|_{\A'}$.  The upper-right corner is just $I'(1,\ldots,1,0) = \colim
  F\big|_{\A_n}$.  Rewriting this, we see that 
  \begin{diagram}
    {\colim F\big|_{\A'\cap \A_n} & \colim F\big|_{\A_n} \\
      \colim F\big|_{\A'} & \colim I'\\};
    \arrowsquare{}{}{}{}
  \end{diagram}
  is a pushout square.  On the other hand, by the $n=2$ case of the lemma
  applied directly to $F$,
  \begin{diagram}
    {\colim F\big|_{\A'\cap \A_n} & \colim F\big|_{\A_n} \\
      \colim F\big|_{\A'} & \colim F\\};
    \arrowsquare{}{}{}{}
  \end{diagram}
  is also a pushout square.  Thus $\colim I' \cong \colim F$, as desired.

  The base case is $n=2$. In particular, we want to show that
  \begin{diagram}
    {\colim F\big|_{\A_1\cap \A_2} & \colim F\big|_{\A_1} \\ \colim
      F\big|_{\A_2} & \colim F\\};
    \arrowsquare{}{}{}{}
  \end{diagram}
  is a pushout square. Let $X$ be the pushout of the upper-left part of the
  square.  We have
  \begin{eqnarray*}
    \colim F &\cong& \coeq \biggl(\coprod_{f\in \A} \dom f \rrto \coprod_{A\in
      \A} A \biggr) \\
    &\cong& \colim
    \left(\begin{inline-diagram}
        {|[text height=1em]|\coprod_{f\in \A_1\cap \A_2} \dom f &|[text height=1em]| \coprod_{A\in \A_1\cap \A_2\phantom{f}} A \\
          |[text height=1em]|\coprod_{f\in \A_1\cup \A_2} \dom f & |[text height=1em]|\coprod_{A\in \A_1 \cup \A_2} A
          \\};
        \to{1-1.east}{1-2.west}_s \to{1-1.mid east}{1-2.mid west}^t
        \to{2-1.east}{2-2.west}_s \to{2-1.mid east}{2-2.mid west}^t
        \to{1-1.260}{2-1.100}_L \to{1-1.280}{2-1.80}^R
        \to{1-2.260}{2-2.100}_L \to{1-2.280}{2-2.80}^R 
      \end{inline-diagram}\right) \\
    &\cong& \coeq
    \left(\begin{inline-diagram}
        { |[text height=1em]| \coeq \biggl(\coprod_{f\in \A_1\cap \A_2} \dom f &
          |[text height=1em]|
          \coprod_{A\in \A_1\cap \A_2} A\biggr) \\
          |[text height=1em]| \coeq \biggl(\coprod_{f\in \A_1\cup \A_2} \dom f &
          |[text height=1em]| \coprod_{A\in \A_1 \cup \A_2} A\biggr)
          \\};
        \to{1-1.east}{1-2.west}_s \to{1-1.mid east}{1-2.mid west}^t
        \to{2-1.east}{2-2.west}_s \to{2-1.mid east}{2-2.mid west}^t
        \to{1-1.south east}{2-1.north east}_L \to{1-2.south west}{2-2.north west}^R
      \end{inline-diagram}\right) \\
    &\cong& \coeq \left(\colim F\big|_{\A_1\cap \A_2} \rrto \colim F\big|_{\A_1}
      \amalg \colim F\big|_{\A_2} \right) \\
    &\cong& X,
  \end{eqnarray*}
  as claimed.  Here, $s$ is the morphism which takes $\dom f$ to itself and $t$
  is the morphism given by $f$ on the component indexed by $f$.  $L$ includes
  $\A_1\cap \A_2$ into $\A_1$, and $R$ includes $\A_1\cap \A_2$ into $\A_2$.

\end{proof}

\bibliographystyle{alpha}
\bibliography{IZ-all}

\begin{thebibliography}{Man12}

\bibitem[BM11]{blumbergmandell11}
Andrew~J. Blumberg and Michael~A. Mandell.
\newblock Derived {K}oszul duality and involutions in the algebraic
  {$K$}-theory of spaces.
\newblock {\em J. Topol.}, 4(2):327--342, 2011.

\bibitem[EM06]{elmendorfmandell}
A.~D. Elmendorf and M.~A. Mandell.
\newblock Rings, modules, and algebras in infinite loop space theory.
\newblock {\em Adv. Math.}, 205(1):163--228, 2006.

\bibitem[GH99]{geisserhesselholt99}
Thomas Geisser and Lars Hesselholt.
\newblock Topological cyclic homology of schemes.
\newblock In {\em Algebraic {$K$}-theory ({S}eattle, {WA}, 1997)}, volume~67 of
  {\em Proc. Sympos. Pure Math.}, pages 41--87. Amer. Math. Soc., Providence,
  RI, 1999.

\bibitem[Goo92]{goodwilliecalcii}
Thomas~G. Goodwillie.
\newblock Calculus. {II}. {A}nalytic functors.
\newblock {\em $K$-Theory}, 5(4):295--332, 1991/92.

\bibitem[Kel82]{kelly82}
Gregory~Maxwell Kelly.
\newblock {\em Basic concepts of enriched category theory}, volume~64 of {\em
  London Mathematical Society Lecture Note Series}.
\newblock Cambridge University Press, Cambridge-New York, 1982.

\bibitem[Lei04]{leinster04}
Tom Leinster.
\newblock {\em Higher operads, higher categories}, volume 298 of {\em London
  Mathematical Society Lecture Note Series}.
\newblock Cambridge University Press, Cambridge, 2004.

\bibitem[Man12]{manzyuk12}
Oleksandr Manzyuk.
\newblock Closed categories vs. closed multicategories.
\newblock {\em Theory Appl. Categ.}, 26:No. 5, 132--175, 2012.

\bibitem[Wal85]{waldhausen83}
Friedhelm Waldhausen.
\newblock Algebraic {$K$}-theory of spaces.
\newblock In {\em Algebraic and geometric topology ({N}ew {B}runswick,
  {N}.{J}., 1983)}, volume 1126 of {\em Lecture Notes in Math.}, pages
  318--419. Springer, Berlin, 1985.

\end{thebibliography}

\end{document}